\documentclass[a4paper,11pt,reqno]{customart}

\usepackage[utf8x]{inputenc}
\usepackage[margin=1.5in]{geometry}
\usepackage{verbatim}
\usepackage{color}
\usepackage[table]{xcolor}
\usepackage{listings}
\usepackage{amssymb,amsfonts,amsthm,amsmath}
\usepackage{url}
\usepackage{enumerate}
\usepackage{todonotes}
\usepackage[all]{xy}
\usepackage{multirow}
\usepackage{tikz}
\usepackage[underline=false]{pgf-umlsd}
\usepackage{stmaryrd}
\usepackage{footnote}
\makesavenoteenv{tabular}
\usepackage{hyperref}
\usepackage{pgfplots}

\makeatletter
\newcommand\footnoteref[1]{\protected@xdef\@thefnmark{\ref{#1}}\@footnotemark}
\makeatother

\hypersetup{
	pdfborder={0 0 0}
}

\definecolor{grey}{rgb}{0.95,0.95,0.95}
\definecolor{green}{rgb}{0.2,0.6,0.4}

\newcommand{\N}{\mathbb{N}}

\renewcommand{\O}{\textbf{0}}

\newcommand{\Psf}{\mathsf{P}}
\newcommand{\Qsf}{\mathsf{Q}}

\newcommand{\biimp}{\leftrightarrow}

\newcommand{\Ccal}{\mathcal{C}}

\newcommand{\Fcal}{\mathcal{F}}

\newcommand{\Vcal}{\mathcal{V}}

\newcommand{\Mcal}{\mathcal{M}}

\newcommand{\Rcal}{\mathcal{R}}

\newcommand{\uh}{{\upharpoonright}}

\renewcommand{\setminus}{\smallsetminus}

\newcommand{\tuple}[1]{\left\langle #1 \right\rangle}

\newcommand{\s}[1]{\ensuremath{\sf{#1}}}

\DeclareMathOperator{\rca}{\s{RCA}_0}
\DeclareMathOperator{\aca}{\s{ACA}}
\DeclareMathOperator{\wkl}{\s{WKL}}
\DeclareMathOperator{\wwkl}{\s{WWKL}}
\DeclareMathOperator{\dnr}{\s{DNR}}

\DeclareMathOperator{\rt}{\s{RT}}
\DeclareMathOperator{\srt}{\s{SRT}}

\DeclareMathOperator{\sads}{\s{SADS}}

\DeclareMathOperator{\emo}{\s{EM}}

\DeclareMathOperator{\opt}{\s{OPT}}

\usetikzlibrary{shapes,arrows}
\usetikzlibrary{decorations.markings}
\definecolor{lightblue}{HTML}{e6e6e6}
\definecolor{lightred}{HTML}{eca6a6}
\definecolor{lightgreen}{RGB}{164,244,140}

\newtheoremstyle{custom}
  {10pt}
  {10pt}
  {\normalfont}
  {}
  {\bfseries}
  {}
  { }
  {}

\theoremstyle{custom}

\usepackage{xcolor}	
\usepackage{soul}

\newtheorem{theorem}{Theorem}[section]
\newtheorem{lemma}[theorem]{Lemma}

\newtheorem{question}[theorem]{Question}

\newtheorem{corollary}[theorem]{Corollary}

\begin{document}

\title{The reverse mathematics of non-decreasing subsequences}
\author{Ludovic Patey}

\begin{abstract}
Every function over the natural numbers has an infinite subdomain on which the function is non-decreasing.
Motivated by a question of Dzhafarov and Schweber, we study the reverse mathematics of variants of this statement.
It turns out that this statement restricted to computably bounded functions is computationally weak
and does not imply the existence of the halting set. On the other hand, we prove that it is not a consequence
of Ramsey's theorem for pairs. This statement can therefore be seen as an arguably natural principle
between the arithmetic comprehension axiom and stable Ramsey's theorem for pairs.
\end{abstract}

\maketitle

\section{Introduction}

A \emph{non-decreasing subsequence} for a function $f : \N \to \N$
is a set~$X \subseteq \N$ such that $f(x) \leq f(y)$ for every pair~$x < y \in X$.
Every function over~$\N \to \N$ admits an infinite non-decreasing subsequence.
Moreover, such a sequence can be computably, but not uniformly obtained from the function~$f$.
Indeed, given $f : \N \to \N$, either there is a value~$y \in \N$
for which the set~$S_y = \{ x \in \N : f(x) = y \}$ is infinite, or for every $y \in \N$,
there is a threshold $t \in \N$ after which $f(x) > y$ for every~$x > t$.
In the former case, the set~$S_y$ is an infinite $f$-computable non-decreasing subsequence for~$f$,
while in the latter case, one can $f$-computably thin out the set~$\N$ to obtain an infinite,
strictly increasing subsequence. This non-uniform argument can be shown to be necessary
by Weihrauch reducing the limited principle of omniscience ($\mathsf{LPO}$) to this statement~\cite{Brattka2011Weihrauch}.

In this paper, we study the reverse mathematics of variants of this statement
by considering various classes of non-computable functions over~$\N \to \N$.
This study is motivated by the following question of Dzhafarov and Schweber in MathOverflow~\cite{Dzhafarov2016Finding}
and taken up by Hirschfeldt in his open questions paper~\cite{Hirschfeldt2016Some}.
A set~$X$ is a \emph{limit non-decreasing subsequence} for a stable
function $f : \N \times \N \to \N$ if it is a non-decreasing subsequence of its limit
function $\tilde{f} : \N \to \N$ defined by $\tilde{f}(x) = \lim_s f(x, s)$.

\begin{quote}
Let $f : \N \times \N \to \N$ be a computable function
such that $f(x, s+1) \leq f(x, s)$ for every~$x, s \in \N$.
Let~$X$ be an infinite limit non-decreasing subsequence for~$f$.
How complicated must such an~$X$ be? In particular, can it avoid computing the halting set?
\end{quote}

Let~$\mathsf{LNS}$ be the statement asserting the existence of an infinite limit non-decreasing subsequence
for any such function $f : \N \times \N \to \N$.
Liang Yu noticed that $\mathsf{LNS}$ implies the existence of a \emph{diagonally non-computable} function,
that is, a function~$h : \N \to \N$ such that~$h(e) \neq \Phi_e(e)$ for every~$e \in \N$.
We identify a natural strengthening of~$\mathsf{LNS}$ that we call~$\mathsf{CNS}$, standing
for ``computably bounded non-decreasing subsequence''. We prove that every computable instance
of~$\mathsf{CNS}$ admits low${}_2$ solutions using the first jump control of Cholak, Jockusch and Slaman~\cite{Cholak2001strength},
and show that $\mathsf{CNS}$ is a computationally weak statement by proving
that it does not imply weak weak K\"onig's lemma ($\wwkl$). On the other hand, $\mathsf{CNS}$ is not a consequence of 
Ramsey's theorem for pairs~($\rt^2_2$) and implies stable Ramsey's theorem for pairs~($\srt^2_2$).
Finally, we separate~$\mathsf{LNS}$ from~$\mathsf{CNS}$ by proving that the former does not imply 
the stable ascending descending sequence principle ($\sads$) in reverse mathematics.

\section{The weakness of non-decreasing subsequences}

First, note that the general statement
of the existence of a non-decreasing sequence for any function over $\N \to \N$
implies the existence of the halting set. Indeed, let~$\mu$ be the \emph{modulus}
function of $\emptyset'$, that is, $\mu(x)$ is the minimal stage~$s$
such that~$\emptyset'_s \uh x = \emptyset' \uh x$,
and let~$f : \N \to \N$ be the function defined by $f(x) = \mu(n) - k - 1$,
where $n$ and~$k$ are the unique solution to the equation $x = k + \sum_{j < n} \mu(j)$
satisfying $k < \mu(n)$.

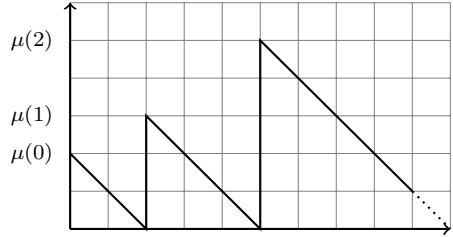
\begin{figure}[htbp]
\begin{center}
\begin{tikzpicture}[x=0.5cm, y=0.5cm]

	\node at (-1, 2) {\scriptsize $\mu(0)$};
	\node at (-1, 3) {\scriptsize $\mu(1)$};
	\node at (-1, 5) {\scriptsize $\mu(2)$};

	\draw[step=0.5cm,gray,very thin] (0,0) grid (10,6);
	\draw[thick,->] (0, 0) -- (10, 0);
  \draw[thick, ->] (0, 0) -- (0, 6);
	\draw[thick] (0, 2) -- (2, 0) -- (2, 3) -- (5, 0) -- (5, 5) -- (9, 1);
	\draw[thick, dotted] (9, 1) -- (10, 0);

\end{tikzpicture}
\end{center}
\end{figure}

The above argument uses finite decreasing sequences
to ensure a sufficient amount of sparsity in the non-decreasing subsequences for~$f$,
to compute fast-growing functions. At first sight, such an argument does not seem to be applicable
to $\mathsf{LNS}$ since the value of~$\tilde{f}(x) = \lim_s f(x, s)$ is bounded by $f(x, 0)$
for every instance~$f$ of~$\mathsf{LNS}$. Therefore, one cannot force the solutions to have a hole
of size more than~$f(x, 0)$ starting from~$x$.
Actually, this computable bounding of the function $\tilde{f}$ is the essential feature
of the weakness of the $\mathsf{LNS}$ statement.
A function~$f : \N \to \N$ is \emph{computably bounded} if it is dominated by a computable function.
Let~$\mathsf{CNS}$ be the statement ``Every computably bounded $\Delta^0_2$ function
admits an infinite non-decreasing subsequence.''
In particular, $\mathsf{CNS}$ generalizes $\mathsf{LNS}$, in that every instance of $\mathsf{LNS}$
can be seen as the $\Delta^0_2$ approximation of a computably bounded function.
As a warm-up, we prove that $\mathsf{CNS}$ admits \emph{cone avoidance}.

\begin{theorem}
Fix a set~$C$ and a set~$A \not \leq_T C$.
For every $C$-computably bounded function~$f : \N \to \N$, there is an infinite set~$G$
non-decreasing for~$f$ such that~$A \not \leq_T G \oplus C$.
\end{theorem}
\begin{proof}
Let~$b : \N \to \N$ be a $C$-computable function bounding~$f$.
Assume that there is no infinite non-decreasing subsequence~$G$ such that~$A \not \leq_T G \oplus C$, 
otherwise we are done. We will construct the set~$G$ using a variant of Mathias forcing.

Our forcing condition are pairs $(F, X)$ where
$F$ is a finite set of integers non-decreasing for~$f$, $X$ is an infinite set such that $\max F < \min X$
and such that $f(x) \leq f(y)$ for every~$x \in F$ and $y \in X$.
We furthermore require that $A \not \leq_T X \oplus C$.
A condition $d = (E, Y)$ extends $c = (F, X)$
if~$d$ Mathias extends $c$, that is, $E \supseteq F$, $Y \subseteq X$ and $E \setminus F \subseteq X$.
A set~$G$ \emph{satisfies} a condition $(F, X)$ if
$F \subseteq G \subseteq F \cup X$.
We start by proving that every sufficiently generic filter for this notion of forcing
yields an infinite set.

\begin{lemma}\label{lem:cone-avoidance-ext}
For every condition~$c = (F, X)$, there is an extension $d = (E, Y)$ of~$c$
such that $|E| > |F|$.
\end{lemma}
\begin{proof}
Pick any~$x \in X$. By strong cone avoidance
of the infinite pigeonhole principle~\cite{Dzhafarov2009Ramseys},
there is an infinite set~$Y \subseteq X \setminus [0, x]$ such that
either~$f$ is constant over~$Y$ for some value smaller than~$f(x)$, or~$f(y) \geq f(x)$ for every~$y \in Y$.
In the former case, $Y$ is an infinite non-decreasing subsequence for~$f$
such that~$A \not \leq_T Y \oplus C$, contradicting our assumption.
In the latter case, the condition~$d = (F \cup \{x\}, Y)$ is the desired extension of~$c$.
\end{proof}

A condition~$c$ \emph{forces} a formula~$\varphi(G)$ if $\varphi(G)$ holds
for every infinite set~$G$ satisfying~$c$. 
We now prove that $A \not \leq_T G \oplus C$ for every set $G$ yielded
by a sufficiently generic filter.

\begin{lemma}\label{lem:cone-avoidance-force}
For every condition~$c = (F, X)$ and every Turing functional~$\Gamma$,
there is an extension~$d$ of~$c$ forcing~$\Gamma^{G \oplus C} \neq A$. 
\end{lemma}
\begin{proof}
For every~$x \in \N$ and~$i < 2$, let~$\Fcal_{x,i}$ be the the $\Pi^{0,X \oplus C}_1$ class of all functions $g : \N \to \N$
dominated by~$b$ such that for every set~$E \subset X$ non-decreasing for~$g$,
\[
\Gamma^{(F \cup E) \oplus C}(x) \uparrow \mbox{ or } \Gamma^{(F \cup E) \oplus C}(n) \downarrow \neq i
\]
Also define $P = \{ (x, i) : \Fcal_{x, i} = \emptyset \}$.
We have three outcomes.
In the first case, $\Fcal_{x, 1-A(x)} \in P$ for some~$x \in \N$. In other words, $\Fcal_{x, 1-A(x)} = \emptyset$.
In particular, $f \not \in \Fcal_{x, 1-A(x)}$, so there is a finite set~$E \subseteq X$
non-decreasing for~$f$, such that $\Gamma^{(F \cup E) \oplus C}(x) \downarrow = 1-A(x)$.
Apply strong cone avoidance avoidance of the infinite pigeonhole principle as in Lemma~\ref{lem:cone-avoidance-ext}
to obtain an infinite set~$Y \subseteq X$ such that
$d = (F \cup E, Y)$ is a valid extension of~$c$ forcing $\Gamma^{G \oplus C}(x) \downarrow \neq A(x)$.

In the second case, there is some~$x \in \N$ such that~$\Fcal_{x, A(x)} \not \in P$.
By the cone avoidance basis theorem~\cite{Jockusch197201}, there is some~$g \in \Fcal_{x, A(x)}$
such that~$A \not \leq_T g \oplus X \oplus C$. We can $g \oplus X$-computably thin out
the set~$X$ to obtain an infinite set~$Y$ non-decreasing for~$g$.
In particular the condition~$d = (F, Y)$ is a valid extension of~$c$
forcing~$\Gamma^{G \oplus C}(x) \uparrow$ or~$\Gamma^{G \oplus C}(x) \downarrow \neq A(x)$. 

In the last case, for every~$x \in \N$ and~$i < 2$, $(x, i) \in P \biimp A(x) = i$.
This case cannot happen, since otherwise $A \leq_T P \leq_T X \oplus C$, contradicting our assumption.
\end{proof}

Let~$\Fcal = \{c_0, c_1, \dots\}$ be a sufficiently generic filter containing $(\emptyset, \omega)$,
where~$c_s = (F_s, X_s)$. The filter~$\Fcal$ yields a unique set~$G = \bigcup_s F_s$.
By Lemma~\ref{lem:cone-avoidance-ext}, the set~$G$ is infinite,
and by definition of a condition, $G$ is non-decreasing for~$f$.
By Lemma~\ref{lem:cone-avoidance-force}, $A \not \leq_T G \oplus C$.
This completes the proof.
\end{proof}

K\"onig's lemma asserts that every infinite, finitely branching tree admits an infinite path.
Weak K\"onig's lemma ($\wkl$) is the restriction of K\"onig's lemma to binary trees.
$\wkl$ plays an important role in reverse mathematics as many statements happen to be equivalent to it~\cite{Simpson2009Subsystems}.
It is therefore natural to compare $\mathsf{CNS}$ and~$\mathsf{LNS}$ to weak K\"onig's lemma.
Actually, we will prove that~$\mathsf{CNS}$ does not imply an even weaker variant of K\"onig's lemma, namely,
weak weak K\"onig's lemma.
A binary tree~$T \subseteq 2^{<\N}$ has \emph{positive measure} if $\lim_s \frac{|\{\sigma \in T : |\sigma| = s\}|}{2^s} > 0$. 
Weak weak K\"onig's lemma ($\wwkl$) is the restriction of~$\wkl$ to binary trees of positive measure.
It can be seen as asserting the existence of a random real, in the sense of Martin-L\"of~\cite{Downey2010Algorithmic}.
Liu~\cite{Liu2015Cone} introduced the notion of constant-bound enumeration avoidance to separate Ramsey's theorem for pairs
from weak weak K\"onig's lemma. We shall use the same notion to separate $\mathsf{CNS}$ from $\wwkl$.

A \emph{$k$-enumeration} (or $k$-enum) of a class $\Ccal \subseteq 2^{\N}$ is a sequence $D_0, D_1, \dots$
such that for each $n \in \N$, $|D_n| \leq k$, $(\forall \sigma \in D_n)|\sigma| = n$
and $\Ccal \cap [D_n] \neq \emptyset$, where $D_n$ is seen as a clopen set of reals in the Cantor space.
A \emph{constant-bound enumeration} (or c.b-enum) of $\Ccal$ is a $k$-enum of $\Ccal$ for some $k \in \N$.
We now prove that $\mathsf{CNS}$ does not imply weak weak K\"onig's lemma over $\rca$.

\begin{theorem}\label{thm:cb-enum-avoidance-cns}
Fix a set~$C$ and a class~$\Ccal \subseteq 2^\N$ with no $C$-computable c.b-enum.
For every $C$-computably bounded function~$f : \N \to \N$, there is an infinite
non-decreasing subsequence~$G$ such that~$\Ccal$ has no $G \oplus C$-computable c.b-enum.
\end{theorem}
\begin{proof}
Let~$b : \N \to \N$ be a $C$-computable function bounding~$f$.
Again, assume that there is no infinite non-decreasing subsequence~$G$ such that~$\Ccal$
has no $G \oplus C$-computable c.b-enum, otherwise we are done.
We will construct the set~$G$ using another variant of Mathias forcing.

Our forcing condition are tuples $(F, X, S)$ where
$F$ is a finite set of integers, $X$ is an infinite set such that $\max F < \min X$,
and $S$ is a finite collection of functions over~$\N \to \N$ dominated by~$b$,
and such that $g(x) \leq g(y)$ for every~$x \in F$, $y \in X$ and~$g \in S$.
We furthermore require that $\Ccal$ has no $X \oplus C$-computable c.b-enum.
A condition $d = (E, Y, T)$ extends $c = (F, X, S)$
if $E \supseteq F$, $Y \subseteq X$, $T \supseteq S$ and 
$E \setminus F$ is a non-decreasing subset of~$X$ for every $g \in S$.
A set~$G$ \emph{satisfies} a condition $(F, X, S)$ if
$F \subseteq G \subseteq F \cup X$ and $G \setminus F$ is non-decreasing for every~$g \in S$.
In particular, every infinite set satisfying the condition $(\emptyset, \omega, \{f\})$
is an infinite non-decreasing sequence for~$f$.

Given a condition~$c = (F, X, S)$, we let $\#(c)$ be the number of functions~$g \in S$
such that~$g$ is not constant over~$X$. 
We now prove that every sufficiently generic filter for this notion of forcing
yields an infinite set.

\begin{lemma}\label{lem:enum-avoidance-ext}
For every condition~$c = (F, X, S)$, there is an extension $d = (E, Y, S)$ of~$c$
such that either~$\#(d) < \#(c)$, or $|E| > |F|$.
\end{lemma}
\begin{proof}
Suppose that~$S = \{g_0, \dots, g_{k-1}\}$.
Pick any~$x \in X$. By iteratively applying strong c.b-enum avoidance
of the infinite pigeonhole principle~\cite{Liu2015Cone},
define a finite sequence $X = X_0 \supseteq X_1 \supseteq \dots \supseteq X_k$
of infinite sets such that for each~$i < k$, $\Ccal$ has no $X_{i+1} \oplus C$-computable c.b-enum
and either there is some~$n < g_i(x)$ such that $g_i(y) = n$ for each~$y \in X_{i+1}$,
or $g_i(y) \geq g_i(x)$ for each~$y \in X_{i+1}$.
If we are in the former case for some~$i < k$, then the condition~$d = (F, X_{i+1}, S)$
is an extension of~$c$ such that~$\#(d) < \#(c)$.
Otherwise, the condition $d = (F \cup \{x\}, X_k, S)$ is the desired extension of~$c$.
\end{proof}

We now prove that $\Ccal$ has no $G \oplus C$-computable c.b-enum for every set $G$ yielded
by a sufficiently generic filter.

\begin{lemma}\label{lem:enum-avoidance-force}
For every condition~$c = (F, X, S)$, every~$k \in \N$ and every Turing functional~$\Gamma$,
there is an extension~$d$ of~$c$ such that either~$\#(d) < \#(c)$, 
or~$d$ forces~$\Gamma^{G \oplus C}$ not to be a valid $k$-enum of~$\Ccal$. 
\end{lemma}
\begin{proof}
Suppose that~$S = \{g_0, \dots, g_{m-1}\}$.
For ease of notation, whenever $\Gamma^{G \oplus C}(n)$ halts,
we will interpret $\Gamma^{G \oplus C}(n)$ as a finite set~$D_n$ of size $k$
such that~$|\sigma| = n$ for every~$\sigma \in D_n$.
For every~$n \in \N$, let~$C_n = \{ \sigma \in 2^n : \Ccal \cap [\sigma] \neq \emptyset \}$.
For every set~$D \subseteq 2^n$,
let~$\Fcal_D$ be the the $\Pi^{0,X \oplus C}_1$ class of all $m$-tuples of functions $\tuple{h_0, \dots, h_{m-1}}$ over~$\N \to \N$
dominated by~$b$, such that~$h_i(y) \geq h_i(x)$ for each~$x \in F$, $y \in X$ and~$i < m$, 
and such that $\Gamma^{(F \cup E) \oplus C}(n) \uparrow$
or $\Gamma^{(F \cup E) \oplus C}(n) \cap D \neq \emptyset$ for every set~$E \subset X$
non-decreasing for every~$h_i$ simultaneously. 
Finally, for each~$n \in \N$, let $P_n = \{D \subseteq 2^n : \Fcal_D \neq \emptyset \}$.
We have three outcomes.

In the first case, $C_n \not \in P_n$ for some~$n \in \N$. In other words, $\Fcal_{C_n} = \emptyset$.
In particular, $\tuple{g_0, \dots, g_{m-1}} \not \in \Fcal_{C_n}$, so there is a finite set~$E \subseteq X$
non-decreasing for every~$g_i$ simultaneously, such that $\Gamma^{(F \cup E) \oplus C}(n) \cap \Ccal = \emptyset$.
Apply strong c.b-enum avoidance of the infinite pigeonhole principle as in Lemma~\ref{lem:enum-avoidance-ext}
to obtain an infinite set~$Y \subseteq X$ such that either~$d = (F, Y, S)$ is an extension of~$c$
satisfying~$\#(d) < \#(c)$, or~$d = (F \cup E, Y, S)$ is a valid extension of~$c$
forcing $[\Gamma^{G \oplus C}(n)] \cap \Ccal = \emptyset$.

In the second case, there is some~$n \in \N$ such that for every $k$-partition $\Vcal_0, \dots, \Vcal_{k-1}$ 
of $P_n$, there is some~$i < k$ such that $\bigcap \Vcal_i = \emptyset$.
For each~$D \in P_n$, pick some~$\langle h^D_0, \dots,\allowbreak h^D_{m-1} \rangle \in \Fcal_D$.
The condition~$d = (F, X, T )$, where $T = S \cup \bigcup_{D \in P_n} \{ h^D_0, \dots, h^D_{m-1} \}$,
is a valid extension of~$c$ forcing $\Gamma^{G \oplus C}(n) \uparrow$.
To see that, suppose that $\Gamma^{G \oplus C}(n) \downarrow = \{\sigma_0, \dots, \sigma_{k-1}\}$,
and let $\Vcal_i = \{D \in P_n : \sigma_i \in D \}$.
We claim that $\Vcal_0, \dots, \Vcal_{k-1}$ forms a $k$-partition of~$P_n$.
Indeed, for any~$D \in P_n$, since~$G$ satisfies $d$, $G \setminus F$ is non-decreasing for 
$h^D_0, \dots, h^D_{m-1}$, so $\{\sigma_0, \dots, \sigma_{k-1}\} \cap D \neq \emptyset$ and $D \in \Vcal_i$ for some~$i < k$.
But then there is some~$i < k$ such that $\bigcap \Vcal_i = \emptyset$,
contradicting $\sigma_i \in \bigcap \Vcal_i$.

In the last case, for every~$n \in \N$, $C_n \in P_n$ and there is a $k$-partition $\Vcal_0, \dots, \Vcal_{k-1}$ 
of~$P_n$ such that $\bigcap \Vcal_i \neq \emptyset$ for each~$i < k$.
In this case, we claim that $\Ccal$ admits an $X \oplus C$-computable $k$-enum, contradicting our assumption.
First note that the set $P_n$ is $X \oplus C$-co-c.e.\ uniformly in~$n$.
Therefore, given $n \in \N$, we can $X \oplus C$-computably find a stage~$s$ and a $k$-partition $\Vcal_0, \dots, \Vcal_{k-1}$ 
of $P_{n,s}$ such that $\bigcap \Vcal_i \neq \emptyset$ for each~$i < k$. Let~$D_n$ be the set
obtained by picking a $\sigma$ in each~$\bigcap \Vcal_i$.
The set~$D_n$ has size~$k$, and $\Ccal \cap [D_n] \neq \emptyset$ since $C_n \in P_{n,s}$.
The sequence $D_0, D_1, \dots$ is an $X \oplus C$-computable $k$-enum of~$\Ccal$.
\end{proof}

Let~$\Fcal = \{c_0, c_1, \dots\}$ be a sufficiently generic filter containing $(\emptyset, \omega, \{f\})$,
where~$c_s = (F_s, X_s, S_s)$. The filter~$\Fcal$ yields a unique set~$G = \bigcup_s F_s$.
By Lemma~\ref{lem:enum-avoidance-ext}, the set~$G$ is infinite,
and by definition of the extension relation, $G$ is non-decreasing for~$f$.
By Lemma~\ref{lem:enum-avoidance-force}, $G \oplus C$ computes no c.b-enum of~$\Ccal$.
This completes the proof of Theorem~\ref{thm:cb-enum-avoidance-cns}.
\end{proof}

\begin{corollary}
$\mathsf{CNS}$ does not imply~$\wwkl$ over~$\rca$.
\end{corollary}
\begin{proof}
Let~$T$ be a computable tree of positive measure whose paths are Martin-L\"of randoms.
By Liu~\cite{Liu2015Cone}, $[T]$ has no computable c.b-enum. Iterate Theorem~\ref{thm:cb-enum-avoidance-cns}
to build a model~$\Mcal$ of~$\mathsf{CNS}$ such that~$[T]$ has no $X$-computable c.b-enum
for any~$X \in \Mcal$. In particular, $T \in \Mcal$, but there is no path through~$T$ in~$\Mcal$,
so~$\Mcal \not \models \wwkl$.
\end{proof}

The statement $\mathsf{CNS}$ enjoys two important properties.
First, any infinite subset of a non-decreasing sequence is itself non-decreasing.
Second, for any function $f : \N \to \N$ and any infinite set~$X \subseteq \N$,
one can find an infinite non-decreasing subsequence~$Y \subseteq X$.
These two features are shared with a whole family of statements coming from Ramsey's theory.
Recall that Ramsey's theorem for $n$-tuples and $k$ colors
($\rt^n_k$) asserts the existence, for every coloring $f : [\N]^n \to k$,
of an infinite \emph{homogeneous set}, that is, a set~$H \subseteq \N$
such that~$[H]^n$ is monochromatic. A coloring $f : [\N]^2 \to k$
is \emph{stable} if $\lim_s f(x, s)$ exists for every~$x \in \N$. 
$\srt^2_k$ is the restriction of $\rt^2_k$ to stable colorings.
Any stable coloring $f : [\N]^2 \to k$ can be seen as the $\Delta^0_2$
approximation of the computably bounded function $\tilde{f} : \N \to \N$
defined by $\tilde{f}(x) = \lim_s f(x, s)$. Moreover, any infinite
non-decreasing subsequence for~$\tilde{f}$ is, up to finite changes,
homogeneous for~$\tilde{f}$, and can be $f \oplus H$-computably thinned out
to obtain an infinite $f$-homogeneous set. 
By Chong, Lempp and Yang~\cite{Chong2010role}, this argument can be formalized in $\rca$,
therefore $\mathsf{CNS}$ implies $\srt^2_k$ over~$\rca$ for every standard~$k \in \N$. 
We will prove in the next section that the converse does not hold. For now, we show that $\mathsf{LNS}$
does not imply $\mathsf{CNS}$ over~$\rca$ using the notion
of preservation of hyperimmunity.

A function $f : \N \to \N$ is \emph{hyperimmune} if it is not dominated by any computable function.
An infinite set is hyperimmune if its principal function is hyperimmune, where the \emph{principal function}
of a set~$X = \{x_0 < x_1 < \dots \}$ is defined by $p_X(n) = x_n$.
A problem~$\Psf$ \emph{admits preservation of hyperimmunity}
if for each set~$C$, each countable collection of $C$-hyperimmune sets~$A_0, A_1, \dots$, 
and each $\Psf$-instance $X \leq_T Z$,
there exists a solution $Y$ to~$X$ such that the $A$'s are $Y \oplus C$-hyperimmune.
The author proved~\cite{Patey2015Iterativea} that weak statements such as the stable version of the ascending descending
principle ($\sads$) do not admit preservation of hyperimmunity, while the Erd\H{o}s-Moser theorem ($\emo$)
does. We shall use this notion to separate $\mathsf{LNS}$ from $\sads$ over~$\rca$.
In particular, this will separate~$\mathsf{LNS}$ from~$\mathsf{CNS}$ since~$\mathsf{CNS}$
implies~$\srt^2_2$, which itself implies~$\sads$ over~$\rca$ (see~\cite{Hirschfeldt2007Combinatorial}).

We have seen that for every computable instance~$f : \N \times \N \to \N$ of~$\mathsf{LNS}$,
its limit function~$\tilde{f}$ is computably bounded, and that this bounding feature is sufficient to obtain cone avoidance.
We will now exploit another property enjoyed by~$\tilde{f}$ to prove that $\mathsf{LNS}$ admits preservation of hyperimmunity.
A function $g : \N \to \N$ is \emph{eventually increasing} if each~$y \in \N$
has finitely many predecessors by~$g$. For every computable instance~$f : \N \times \N \to \N$
of~$\mathsf{LNS}$ with no computable solution, its limit function~$\tilde{f}$ must be eventually increasing,
otherwise the set~$\{x : \tilde{f}(x, s) = y \}$ would be an infinite, computable
non-decreasing subsequence for~$\tilde{f}$ for the least $y$ with
infinitely many predecessors by~$\tilde{f}$. Let~$\mathsf{ICNS}$ be the restriction of~$\mathsf{CNS}$
to eventually increasing functions. We will now prove that~$\mathsf{ICNS}$,
and therefore~$\mathsf{LNS}$, admits preservation of hyperimmunity.

\begin{theorem}\label{thm:lns-preservation-hyperimmunity}
Fix a set~$C$ and a countable sequence~$A_0, A_1, \dots$ of $C$-hyperimmune sets.
For every $C$-computably bounded, eventually increasing function~$f : \N \to \N$, there is an infinite
non-decreasing subsequence~$G$ such that the~$A$'s are $G \oplus C$-hyperimmune.
\end{theorem}
\begin{proof}
Let~$b : \N \to \N$ be a $C$-computable function bounding~$f$.
As usual, assume that there is no infinite set~$G$ non-decreasing for~$f$
such that the~$A$'s are $G \oplus C$-hyperimmune, otherwise we are done.
We will build the set~$G$ by a variant of Mathias forcing.

A condition is a tuple~$(F, X)$ where~$F$ is a finite set of integers
non-decreasing for~$f$, $X$ is an infinite set such that~$\max F < \min X$,
$f(x) \leq f(y)$ for every~$x \in F$, $y \in X$, and the $A$'s are $X \oplus C$-hyperimmune.
The notions of condition extension and of set satisfaction inherit from Mathias forcing.
We again prove that every sufficiently generic filter for this notion of forcing
yields an infinite set.

\begin{lemma}\label{lem:hyp-preservation-ext}
For every condition~$c = (F, X)$, there is an extension $d = (E, Y)$
such that~$|E| > |F|$.
\end{lemma}
\begin{proof}
Pick any~$x \in X$ and let~$Y = \{ y \in X : y > x \wedge f(y) \geq f(x) \}$.
The set~$Y$ is obtained from~$X$ by removing finitely many elements
since~$f$ is eventually increasing,
so the~$A$'s are $Y \oplus C$-hyperimmune.
The condition~$(F \cup \{x\}, Y)$ is the desired extension of~$c$.
\end{proof}

Next, we prove that every sufficiently generic filter yields
a set~$G$ such that the~$A$'s are $G \oplus C$-hyperimmune.

\begin{lemma}\label{lem:hyp-preservation-force}
For every condition~$c = (F, X)$, every Turing functional~$\Gamma$
and every~$i \in \N$, there is an extension forcing~$\Gamma^{G \oplus C}$ not to dominate $p_{A_i}$.
\end{lemma}
\begin{proof}
Let~$h$ be the partial $X \oplus C$-computable function which on input~$x$
searches for a finite set of integers~$U$ such that for every function $g : \N \to \N$
bounded by~$b$, there is a finite set~$E \subseteq X$ non-decreasing for~$g$
such that~$\Phi^{(F \cup E) \oplus C}_e(x) \downarrow \in U$. If such a set~$U$ is found,
$f(x) = \max U$, otherwise~$f(x) \uparrow$. We have two cases.

Case 1: $h$ is total. By $X \oplus C$-hyperimmunity of~$A_i$, there
	is some~$x$ such that~$h(x) < p_{A_i}(x)$. Let~$U$ be the finite set witnessing~$h(x) \downarrow$.
	In particular, taking~$g = f$, there is a finite set~$E \subseteq X$
	non-decreasing for~$f$ such that~$\Phi^{(F \cup E) \oplus C}_e(x) \downarrow \in U$.
	By removing finitely many elements from~$X$, we obtain a set~$Y$ such that 
	$(F \cup E, Y)$ is a valid extension of~$c$ forcing~$\Phi^{G \oplus C}_e(x) \downarrow < p_{A_i}(x)$.

Case 2: there is some~$x$ such that~$h(x) \uparrow$. Let~$\Ccal$
	be the~$\Pi^{0,X \oplus C}_1$ class of functions $g : \N \to \N$ bounded by~$b$
	such that for every finite set~$E \subseteq X$ non-decreasing for~$g$,
	$\Phi^{(F \cup E) \oplus C}_e(x) \uparrow$. 
	By compactness, $\Ccal \neq \emptyset$, so
	by preservation of hyperimmunity of~$\wkl$, there exists some $g \in \Ccal$
	such that the~$A$'s are $g \oplus X \oplus C$-hyperimmune.
	We can $g \oplus X$-computably thin out the set~$X$ to obtain an infinite set~$Y \subseteq X$
	non-decreasing for~$g$. The condition~$(F, Y)$ is an extension of~$c$
	forcing~$\Phi^{G \oplus C}_e(x) \uparrow$.
\end{proof}

Let~$\Fcal = \{c_0, c_1, \dots\}$ be a sufficiently generic filter containing $(\emptyset, \omega)$,
where~$c_s = (F_s, X_s)$. The filter~$\Fcal$ yields a unique set~$G = \bigcup_s F_s$.
By Lemma~\ref{lem:hyp-preservation-ext}, the set~$G$ is infinite,
and by definition of a condition, $G$ is non-decreasing for~$f$.
By Lemma~\ref{lem:hyp-preservation-force}, the~$A$'s are $G \oplus C$-hyperimmune.
This completes the proof of Theorem~\ref{thm:lns-preservation-hyperimmunity}.
\end{proof}

\begin{corollary}
$\mathsf{ICNS} \wedge \emo \wedge \wkl$ does not imply $\sads$ over~$\rca$.
\end{corollary}
\begin{proof}
By Theorem~\ref{thm:lns-preservation-hyperimmunity}, by~\cite{Patey2015Iterativea} and by the
hyperimmune-free basis theorem~\cite{Jockusch197201}, $\mathsf{ICNS}$, $\emo$ and~$\wkl$
admit preservation of hyperimmunity, while~$\sads$ does not.
One can therefore build an $\omega$-model of~$\mathsf{ICNS} \wedge \emo \wedge \wkl$
in which~$\sads$ does not hold.
\end{proof}

\section{The strength of non-decreasing subsequences}

We continue our study of the strength of the non-decreasing statements
by considering their ability to compute functions not dominated
by some classes of functions. Since stable Ramsey's theorem for pairs
is computably reducible to $\mathsf{CNS}$, there is a computable instance of~$\mathsf{CNS}$
whose solutions are all of hyperimmune degree. We now prove that the same property holds
for~$\mathsf{LNS}$.

\begin{theorem}
There is a computable function~$f : \N \times \N \to \N$
such that $f(x, s+1) \leq f(x, s)$ for every~$x, s \in \N$
and such that every infinite limit non-decreasing subsequence for~$f$
is hyperimmune.
\end{theorem}
\begin{proof}
We will construct the function~$f$ so that $p_H$ is hyperimmune
for every infinite limit non-decreasing subsequence~$H$ for~$f$.
We want to satisfy the following requirements for every~$e \in \N$.
\begin{quote}
$\Rcal_e$ : If $\Phi_e$ is total and increasing, then 
$\Phi_e(x_0) \downarrow = x_1$ and~$\Phi_e(x_1) \downarrow = x_2$
for some $x_0 < x_1 < x_2 \in \N$ such that
$\lim_s f(x, s) > \lim_s f(y, s)$ for each~$x \in [x_0, x_1)$ and~$y \in [x_1, x_2)$.
\end{quote}
Indeed, given an infinite limit non-decreasing subsequence~$H$ for~$f$
let~$\Phi_e$ be any computable increasing function. By~$\Rcal_e$,
either $H \cap [x_0, x_1) = \emptyset$ or~$H \cap [x_1, x_2) = \emptyset$.
In the former case, $p_H(x_0) \geq x_1 = \Phi_e(x_0)$, while
in the latter case $p_H(x_1)  \geq x_2 = \Phi_e(x_1)$.

The overall construction is a finite injury priority argument.
The local strategy for~$\Rcal_e$ \emph{requires attention} at stage~$s$
if $\Phi_e(x_0) \downarrow = x_1$ and~$\Phi_e(x_1) \downarrow = x_2$
for some~$x_0 < x_1 < x_2 < s$ such that~$f(x, s) > e$ for each~$x \in [x_0, x_2)$
and such that no value in~$[x_0, x_2)$ is restrained by a strategy of higher priority.
The strategy for~$\Rcal_e$ commits $f(x, t)$ to be equal to~$e$ for every~$x \in [x_1, x_2)$ and any~$t \geq s$.
It then puts restrains on every value in~$[x_0, x_2)$ and is declared \emph{satisfied}.
If at a later stage, some strategy of higher priority restrains some value in~$[x_0, x_2)$, then
the strategy for~$\Rcal_e$ is \emph{injured} and starts over, releasing all its restrains.

The global construction works as follows.
At stage $0$, $f$ is the empty function.
Suppose that at stage~$s$, the function~$f$ is defined over~$[0, s)^2$.
If some strategy requires attention, then pick the one of highest priority
and run it.
In any case, set~$f(x, s) = e$ for every strategy~$\Rcal_e$ which has committed
such an assignment. Then set~$f(x, s) = f(x, s-1)$ for every~$x < s$ which has not been assigned yet,
and~$f(s, t) = s$ for every~$t \leq s$.
Then go to the next stage. This finishes the construction. We now turn to the verification.

First notice that each strategy acts finitely often,
and therefore that each strategy is injured finitely many times.
Moreover, notice that $f(x, s+1) \leq f(x, s)$ for every~$x, s \in \N$
since when~$f(x, s+1) \neq f(x, s)$, this is caused by a strategy 
which made its value decrease.
We claim that each strategy~$\Rcal_e$ is eventually satisfied.
To see that, let~$\Phi_e$ be a total increasing function
and let~$s_0 > e$ be a stage after which no strategy of higher priority ever acts.
By construction, $f(x, s) > e$ for every~$x, s \geq s_0$.
Therefore at some later stage~$s_1$, there will be some~$x_0 < x_1 < x_2 < s_1$
such that~$\Phi_e(x_0) \downarrow = x_1$ and~$\Phi_e(x_1) \downarrow = x_2$.
In particular, $f(x, s) > e$ for each~$x \in [x_0, x_2)$, so the strategy for~$\Rcal_e$
will require attention and will be satisfied since no strategy of higher priority acts.
This completes the verification.
\end{proof}

We will now prove that $\rt^2_2 \wedge \wkl$
does not imply~$\mathsf{CNS}$ over~$\rca$ using the notion of hypersurjectivity.
A formula $\varphi(U)$, where~$U$ is a finite coded set parameter,
is \emph{essential} if for every~$x \in \N$, there is some finite set~$A > x$
such that~$\varphi(A)$ holds. Given a set~$C$ and an infinite set~$L \subseteq \N$, a 
function~$f : \N \to \N$ is \emph{$C$-hypersurjective for~$L$}
if for every essential $\Sigma^{0,C}_1$ formula~$\varphi(U)$
and every~$y \in L$, $f(A) = \{y\}$ for some finite set~$A$ such that $\varphi(A)$ holds.
We say that~$f$ is $C$-hypersurjective if it is $C$-hypersurjective for some infinite set~$L \subseteq \N$.
A problem~$\Psf$ \emph{admits preservation of hypersurjectivity}
if for each set~$C$, each function~$f : \N \to \N$
which is $C$-hypersurjective, and each $\Psf$-instance~$X \leq_T C$,
there exists a solution~$Y$ to~$X$ and  such that~$f$ is $Y \oplus C$-hypersurjective.

\begin{theorem}
$\mathsf{CNS}$ does not admit preservation of hypersurjectivity.
\end{theorem}
\begin{proof}
We will build a $\Delta^0_2$ function~$f : \N \to \N$ hypersurjective for~$\N$,
such that~$f(x) \leq x$ for every~$x \in \N$.
We first claim that for every infinite set~$H$
non-decreasing for~$f$ and every infinite set~$L \subseteq \N$, $f$ is not $H$-hypersurjective for~$L$. 
Therefore, $f$ is a computable instance of~$\mathsf{CNS}$
whose solutions do not preserve its own hypersurjectivity.

Suppose for the sake of contradiction that $f$ is $H$-hypersurjective for some infinite set~$L \subseteq \N$.
Let~$y$ be the first element of~$L$. We have two cases.
First, suppose that $f(x) \leq y$ for every~$x \in H$. Let~$\varphi(U)$
be the $\Sigma^{0,H}_1$ formula which holds if~$U$ is a non-empty subset of~$H$. 
The formula~$\varphi(U)$ is essential since~$H$ is infinite.
However, let~$z$ be the second element of~$L$. There is no finite set~$A$
such that~$f(A) = \{z\}$ and $\varphi(A)$, otherwise there would be some~$x \in A \subseteq H$
such that~$f(x) = z > y$. This contradicts $H$-hypersurjectivity of~$f$ for~$L$.
Second, suppose that there is some~$x \in H$ such that~$f(x) > y$.
Let~$\psi(U)$ be the $\Sigma^{0,H}_1$ formula which holds if~$U$ is a non-empty subset of~$H \setminus [0, x]$. 
The formula~$\psi(U)$ is again essential since~$H$ is infinite. 
However, if there is a finite set~$A$ such that~$f(A) = \{y\}$ and~$\psi(A)$ holds, then
there is some~$z \in A \subseteq H \setminus [0, x]$ such that~$f(z) = y$.
In particular, $x < z$ and~$f(x) > f(z)$ which contradicts the fact that $H$ is non-decreasing for~$f$.
Therefore $f$ is not $H$-hypersurjective.

We now build the function~$f : \N \to \N$ by the finite extension method
in a $\Delta^0_2$ construction. Fix an enumeration
$\varphi_0(U), \varphi_1(U), \dots$ of all $\Sigma^0_1$ formulas.
Start at stage~$0$ with the empty function~$f$.
Suppose that at stage~$s = \tuple{y, e}$, the function~$f$ is defined over
some domain~$[0, m)$. Decide in $\emptyset'$ whether there is a finite set~$A \geq m$ such that~$\varphi_e(A)$ holds.
If so, set~$f(x) = y$ for every~$x \in [m, \max A]$, otherwise set~$f(m) = 0$.
In both case, go to the next stage. This completes the construction.
\end{proof}

Before proving that~$\rt^2_2$ admits preservation of hypersurjectivity,
we first need to prove that so does $\wkl$ for any fixed~$L \subseteq \N$. 
Indeed, the latter will be used in the proof of the former.
The proof of the following theorem is a slight modification of the proof of Theorem~14 in~\cite{Patey2015Iterativea}.

\begin{theorem}\label{thm:wkl-preservation-hypersurjectivity}
$\wkl$ admits preservation of hypersurjectivity for any fixed~$L$.
\end{theorem}
\begin{proof}
Fix a set~$C$, let~$f : \N \to \N$ be a function $C$-hypersurjective for~$L$,
and let~$T \subseteq 2^{<\omega}$ be a $C$-computable infinite binary tree.
We construct an infinite decreasing sequence of computable subtrees $T = T_0 \supseteq T_1 \supseteq \dots$
such that for every path $P$ through $\bigcap_s T_s$, $f$ is $P \oplus C$-hypersurjective for~$L$.
Note that the intersection~$\bigcap_s T_s$ is non-empty since the $T$'s are infinite trees.
More precisely, if we interpret $s$ as a tuple~$\tuple{y, \varphi}$ where $y \in L$
and $\varphi(G,U)$ is a $\Sigma^{0,C}_1$ formula, we want to satisfy the following requirement.

\begin{quote}
$\Rcal_s$ : For every path~$P$ through~$T_{s+1}$, either $\varphi(P, U)$ is not essential,
or~$\varphi(P, A)$ holds for some finite set~$A$ such that~$f(A) = \{y\}$.
\end{quote}

At stage~$s = \tuple{y, \varphi}$, given some infinite, computable binary tree~$T_s$, 
define the $\Sigma^{0,C}_1$ formula
\[
\psi(U) = (\exists n)(\forall \tau \in T_s \cap 2^n)(\exists \tilde{A} \subseteq U)\varphi(\tau, \tilde{A})
\]
We have two cases.
In the first case, $\psi(U)$ is not essential with some witness~$t$. By compactness,
the following set is an infinite $C$-computable subtree of~$T_s$:
\[
T_{s+1} = \{ \tau \in T_s : (\forall A > t)\neg \varphi(\tau, A) \}
\]
The tree $T_{s+1}$ has been defined so that $\varphi(P, U)$
is not essential for every~$P \in [T_{s+1}]$.
In the second case, $\psi(U)$ is essential. By $C$-hypersurjectivity of~$f$ for~$L$,
there is a finite set~$A$ such that~$\psi(A)$ holds and~$f(A) = \{y\}$.
We claim that for every path~$P \in [T_s]$,
$\varphi(P, \tilde{A})$ holds for some set~$\tilde{A}$ such that~$f(\tilde{A}) = \{y\}$.
Fix some path~$P \in [T_s]$. Unfolding the definition of~$\psi(A)$, there is some~$n$ such that $\varphi(P \uh n, \tilde{A})$ holds
for some set~$\tilde{A} \subseteq A$. By continuity, $\varphi(P, \tilde{A})$ holds.
Moreover, $f(\tilde{A}) = \{y\}$ since $f(A) = \{y\}$
Set~$T_{s+1} = T_s$ and go to the next stage.
This completes the proof of Theorem~\ref{thm:wkl-preservation-hypersurjectivity}.
\end{proof}

We are now ready to prove that Ramsey's theorem for pairs
admits preservation of hypersurjectivity.

\begin{theorem}
$\rt^2_2$ admits preservation of hypersurjectivity.
\end{theorem}
\begin{proof}
Let~$C$ be a set and $g : \N \to \N$ be a function $C$-hypersurjective for some infinite set~$L \subseteq \N$.
Fix a $C$-computable coloring $f : [\N]^2 \to 2$. As usual, assume that there is no infinite $f$-homogeneous set~$H$
such that $g$ is $H \oplus C$-hypersurjective, otherwise we are done.
We will build two infinite sets~$G_0, G_1$, $f$-homogeneous for color 0 and 1, respectively, 
and such that $g$ is either~$G_0 \oplus C$-hypersurjective, or~$G_1 \oplus C$-hypersurjective.

We will use forcing conditions $(F_0, F_1, X)$, where $F_0$ and~$F_1$ are finite sets of integers,
$X$ is an infinite set such that~$\max(F_0, F_1) < \min X$ and for every $i < 2$ and every~$x \in X$,
$F_i \cup \{x\}$ is $f$-homogeneous for color~$i$. We furthermore impose that~$g$ is $X \oplus C$-hypersurjective for~$L$.
A condition~$d = (E_0, E_1, Y)$ extends $c = (F_0, F_1, X)$ if $(E_i, Y)$ Mathias extends $(F_i, X)$ for each~$i < 2$.
A pair of sets~$G_0, G_1$ satisfies a condition~$(F_0, F_1, X)$ if for each~$i < 2$, $G_i$ if $f$-homogeneous for color~$i$
and satisfies the Mathias condition~$(F_i, X)$. Again, we start by proving that every infinite filter yields two infinite sets.

\begin{lemma}\label{lem:hypersurjectivity-preservation-ext}
For every condition~$c = (F_0, F_1, X)$ and every~$i < 2$, 
there is an extension~$d = (E_0, E_1, X)$ of $c$ such that~$|E_i| > |F_i|$.
\end{lemma}
\begin{proof}
For every~$x \in X$, let~$S_x = \{ y \in X : y > x \wedge f(x, y) = i \}$.
If $S_x$ is finite for every $x \in X$, then one can $X \oplus C$-computably thin out
the set~$X$ to obtain an infinite set $f$-homogeneous for color~$1-i$, contradicting our assumption.
Therefore, there is some~$x \in X$ such that~$S_x$ is infinite. The condition~$(E_0, E_1, S_x)$
where~$E_i = F_i \cup \{x\}$ and $E_{1-i} = F_{1-i}$ is the desired extension of~$c$.
\end{proof}

Fix an enumeration $\varphi_0(G, U), \varphi_1(G, U), \dots$ of all $\Sigma^{0, C}_1$ formulas.
We will now ensure the following disjunctive requirements for each~$e_0, e_1 \in \N$ and~$y \in L$.
\[
\Rcal_{e_0, e_1, y} : \Rcal^{G_0}_{e_0,y} \vee \Rcal^{G_1}_{e_1,y}
\]
where~$\Rcal^G_{e,y}$ is the statement ``If $\varphi_e(G, U)$ is essential, then $\varphi(G, A)$ holds
for some finite set~$A$ such that~$g(A) = \{y\}$''.
If all the disjunctive requirements are satisfied for some pair of sets~$G_0, G_1$,
then there will be a 2-partition~$L_0 \cup L_1 = L$ such that $g$ is $G_i \oplus C$-hypersurjective for~$L_i$
for each~$i < 2$. Among $L_0$ and~$L_1$, at least one must be infinite. We will then pick the corresponding~$G_i$.
We say that a condition~$c$ forces~$\Rcal_{e_0, e_1, y}$ if it holds for every pair of sets~$G_0, G_1$
satisfying~$c$.

\begin{lemma}\label{lem:hypersurjectivity-preservation-force}
For every condition~$c = (F_0, F_1, X)$, every pair of indices~$e_0, e_1 \in \N$
and every $y \in L$, there is an extension~$d$ of~$c$ forcing~$\Rcal_{e_0, e_1, y}$.
\end{lemma}
\begin{proof}
Let~$\psi(U)$ be the~$\Sigma^{0,X \oplus C}_1$ formula which holds
if there is a finite set~$H \subseteq X$ such that for every 2-partition~$H_0 \cup H_1 = H$,
there is some~$i < 2$, some finite set~$U_i \subseteq U$ and some set~$E \subseteq H_i$
$f$-homogeneous for color~$i$ such that~$\varphi_{e_i}(F_i \cup E, U_i)$ holds.
We have two cases.

Case 1: the formula~$\psi(U)$ is not essential, with witness~$x \in \N$.
Let~$\Ccal$ be the $\Pi^{0, X \oplus C}_1$ class of all sets~$H_0 \oplus H_1$
such that $H_0 \cup H_1 = X$ and for every~$i < 2$, every finite set~$U_i > x$ and every finite set~$E \subseteq H_i$
$f$-homogeneous for color~$i$, $\varphi_{e_i}(F_i \cup E, U_i)$ does not hold.
Since there is not finite set~$U > x$ such that~$\varphi(U)$ holds, then by a compactness
argument~$\Ccal$ is non-empty. By preservation of hypersurjectivity of~$\wkl$ for~$L$ 
(Theorem~\ref{thm:wkl-preservation-hypersurjectivity}), there is some $H_0 \oplus H_1 \in \Ccal$
such that $g$ is $H_0 \oplus H_1 \oplus X \oplus C$-hypersurjective for~$L$.
Let~$i < 2$ be such that~$H_i$ is infinite. The condition~$(F_0, F_1, H_i)$
is an extension of~$c$ forcing~$\Rcal^{G_i}_{e_i, y}$, hence~$\Rcal_{e_0, e_1, y}$.

Case 2: the formula~$\psi(U)$ is essential. By $X \oplus C$-hypersurjectivity of~$g$ for~$L$,
$\psi(A)$ holds for some finite set~$A$ such that~$g(A) = \{y\}$.
Let~$H \subseteq X$ be the finite set witnessing that~$\psi(A)$ holds.
Every~$z \in X$, induces a 2-partition~$H_0 \cup H_1 = H$ defined by
$H_i = \{ x \in H : f(x, z) = i \}$. Since there are finitely many 2-partitions of~$H$,
there is a 2-partition~$H_0 \cup H_1 = H$ such that the set
\[
Y = \{ z : X : z > \max H \wedge (\forall i < 2)(\forall x \in H_i)f(x, z) = i \}
\]
is infinite. In particular, there is some~$i < 2$ and some set~$E \subseteq H_i$
$f$-homogeneous for color~$i$ such that~$\varphi_{e_i}(F_i \cup E, A_i)$ holds
for some set~$A_i \subseteq A$. The condition~$(E_0, E_1, Y)$
defined by~$E_i = F_i \cup E$ and~$E_{1-i} = F_{1-i}$ is an extension of~$c$
forcing~$\Rcal^{G_i}_{e_i, y}$, hence $\Rcal_{e_0, e_1, y}$.
\end{proof}

Let~$\Fcal = \{c_0, c_1, \dots\}$ be a sufficiently generic filter containing $(\emptyset, \emptyset, \omega)$,
where $c_s = (F_{0,s}, F_{1,s}, X_s)$. The filter~$\Fcal$ yields a unique pair of sets~$G_0 = \bigcup_s F_{0,s}$
and~$G_1 = \bigcup_s F_{1,s}$.
By Lemma~\ref{lem:hypersurjectivity-preservation-ext}, both~$G_0$ and~$G_1$ are infinite.
By Lemma~\ref{lem:hypersurjectivity-preservation-force}, there is some~$i < 2$ and some infinite set~$L_i \subseteq L$
such that $g$ is $G_i \oplus C$-hypersurjective for~$L$.
This completes the proof.
\end{proof}

\begin{corollary}
$\rt^2_2 \wedge \wkl$ does not imply~$\mathsf{CNS}$ over~$\rca$.
\end{corollary}

\section{Low${}_2$ non-decreasing subsequences}

Cholak, Jockusch and Slaman~\cite{Cholak2001strength} proved that every computable instance of~$\rt^2_2$
admits a low${}_2$ solution. In this section, we prove that the same property holds for~$\mathsf{CNS}$.
Given two sets~$X$ and~$A$, an integer~$e \in \N$ is a \emph{$\Delta^{0,X}_2$ index} of~$A$ if $\Phi^{X'}_e = A$.
Similarly, $e$ is an \emph{$X$-jump index} of~$A$ if $\Phi^X_e = A'$.
A function~$f : \N \to \{0, 1\}$ is \emph{$X$-dnc${}_2$} if~$f(e) \neq \Phi^X_e(e)$ for every~$e$.
The following theorem is obtained by looking at the uniformity
of the first jump control of Cholak, Jockusch and Slaman~\cite{Cholak2001strength}.

\begin{theorem}\label{thm:first-jump-d22}
There are two computable functions~$h_0, h_1 : \N \to \N$
such that for every set~$C$ and every $C'$-dnc${}_2$ function~$f$,
if~$e$ is the $\Delta^{0,C}_2$ index of a set~$A$,
then either~$h_0(e)$ is an $f$-jump index of~$Y_0 \oplus C$, where~$Y_0$ is an infinite subset of~$A$,
or~$h_1(e)$ is an $f$-jump index of~$Y_1 \oplus C$, where~$Y_1$ is an infinite subset of~$\overline{A}$.
\end{theorem}
\begin{proof}
Fix a set~$C$ and let~$f$ be a $C'$-dnc${}_2$ function
and $e$ be an index of a $\Delta^{0,C}_2$ set~$A$.
We will describe an $f$-computable construction of a set~$G$
such that for every pair~$e_0, e_1 \in \N$,
either~$((G \cap A) \oplus C)'(e_0)$, or~$((G \cap \overline{A}) \oplus C)'(e_1)$ is decided.
We work with Mathias conditions~$(F, X)$ where~$X$ is low over~$C$.
An \emph{index} of such a condition~$c = (F, X)$ is a code $\tuple{F, i}$ such that~$\Phi_i^{C'} = (X \oplus C)'$.

To simplify our notation, we let~$A_0 = A$ and~$A_1 = \overline{A}$.
Given some~$i < 2$ and some~$e_i \in \N$, a condition~$c = (F, X)$ \emph{decides} $((G \cap A_i) \oplus C)'(e_i)$ if
either~$\Phi_{e_i}^{(F \cap A_i) \oplus C}(e_i) \downarrow$, or~$\Phi_{e_i}^{((F \cap A_i) \cup H) \oplus C}(e_i) \uparrow$
for every set~$H \subseteq X$. Note that $H$ is not necessarily included in~$A_i$. This precision will be used
in Lemma~\ref{lem:first-jump-d22-finite}.

\begin{lemma}[Lemma 4.6 in \cite{Cholak2001strength}]\label{lem:first-jump-d22-force}
Given a condition~$c = (F, X)$ and a pair of indices~$e_0, e_1 \in \N$,
there is an extension~$d$ of~$c$ deciding either~$((G \cap A) \oplus C)'(e_0)$, or~$((G \cap \overline{A}) \oplus C)'(e_1)$.
Furthermore, and index of~$d$ may be $f$-computably computed from~$e_0, e_1$ and an index of~$c$,
and one can $f$-computably decide which case applies. 
\end{lemma}

Using Lemma~\ref{lem:first-jump-d22-force}, build an infinite
$f$-computable decreasing sequence of conditions~$(\emptyset, \omega) = (F_0, X_0) \geq (F_1, X_1) \geq \dots$
such that for every~$s = \tuple{e_0, e_1}$,
$(F_{s+1}, X_{s+1})$ decides either~$((G \cap A) \oplus C)'(e_0)$, or~$((G \cap \overline{A}) \oplus C)'(e_1)$.
Unlike the original construction~\cite{Cholak2001strength}, we do not interleave
requirements to ensure that both~$G \cap A$ and $G \cap \overline{A}$ are infinite,
and indeed, it is not possible since we cannot uniformly decide whether there is a low solution or not. 
Thankfully, the infinity requirements are already ensured by the decision process,
as shows the following lemma.

\begin{lemma}\label{lem:first-jump-d22-finite}
If~$G \cap A_i$ is finite, then for some~$e_i \in \N$,
there is no stage~$s$ at which $(F_{s+1}, X_{s+1})$ decides~$((G \cap A_i) \oplus C)'(e_i)$.
\end{lemma}
\begin{proof}
Let~$k = |G \cap A_i|$, and let~$e_i \in \N$ be such that for every set~$H$,
$\Phi_{e_i}^{H \oplus C}(e_i) \downarrow$ if and only if~$H$ contains at least~$k+1$ elements.
Suppose there is a stage~$s$ at which $(F_{s+1}, X_{s+1})$ decides~$((G \cap A_i) \oplus C)'(e_i)$.
By definition, $\Phi_{e_i}^{(F \cap A_i) \oplus C}(e_i) \downarrow$, 
or~$\Phi_{e_i}^{((F \cap A_i) \cup H) \oplus C}(e_i) \uparrow$ for every set~$H \subseteq X$.
The former does not hold since~$|F \cap A_i| = k$, and neither does the latter
since $\Phi_{e_i}^{((F \cap A_i) \cup H) \oplus C}(e_i) \downarrow$ for any infinite set~$H \subseteq X$.
\end{proof}

For each~$i < 2$, let $h_i(e)$ be the Turing index of the $f$-algorithm
which on input~$e_i$, $f$-computably runs the construction until it finds some stage~$s$
at which~$(F_{s+1}, X_{s+1})$ decides~$((G \cap A) \oplus C)'(e_i)$. 
If such stage is found, the algorithm outputs the answer, otherwise it does not terminate.
We claim that one of the two following holds:
\begin{itemize}
	\item[(a)] $h_0(e)$ is an $f$-jump index of~$(G \cap A) \oplus C$ and~$G \cap A$ is infinite;
	\item[(b)] $h_1(e)$ is an $f$-jump index of~$(G \cap \overline{A}) \oplus C$ and $G \cap \overline{A}$ is infinite.
\end{itemize}
If case (a) does not hold, then either~$G \cap A$ is finite,
or the algorithm of~$h_0(e)$ is not total, and by Lemma~\ref{lem:first-jump-d22-finite},
the former implies the latter. Moreover, if the algorithm of~$h_0(e)$ is not total,
then by the usual pairing argument, the algorithm of~$h_1(e)$ is total,
and by the contrapositive of Lemma~\ref{lem:first-jump-d22-finite}, $G \cap \overline{A}$
is infinite, so case (b) holds.
This completes the proof.
\end{proof}

We are now ready to prove the main theorem of the section.

\begin{theorem}\label{thm:first-jump-cns}
Fix a set~$C$ and a set~$P \gg C'$.
For every~$C$-computable instance of~$\mathsf{CNS}$,
there is an infinite non-decreasing subsequence~$G$ such that~$(G \oplus C)' \leq_T P$.
\end{theorem}
\begin{proof}
Let~$f : \N \to \N$ be a $\Delta^{0,C}_2$ function $C$-computably bounded
by some function~$b : \N \to \N$.
Let $h_0, h_1, \dots$ be a uniformly $P$-computable sequence of functions
such that~$h_0 \equiv_T C'$, and $h_{i+1}$ is~$h_i$-dnc${}_2$ for every~$i \in \N$.
Assume that there is no infinite set~$G$ over which~$f$ is constant, and such that~$(G \oplus C)' \leq_T P$, otherwise we are done.
We will build our set~$G$ by a $P$-computable construction
using variants of Mathias conditions.

An \emph{$h_i$-condition} is a tuple~$(F, X, S)$
where~$(F, X)$ is a Mathias condition, 
$g(x) \leq g(y)$ for every~$x \in F$, $y \in X$ and~$g \in S \cup \{f\}$,
and $S$ is a finite collection of functions bounded~$b$,
such that $(X \oplus S \oplus C)' \leq_T h_i$.
A $h_j$-condition~$d = (E, Y, T)$ extends an $h_i$-condition~$c = (F, X, S)$
if~$j \geq i$,  $E \supseteq F$, $Y \subseteq X$, $T \supseteq S$ and 
$E \setminus F$ is a non-decreasing subset of~$X$ for every $g \in S \cup \{f\}$.
A set~$G$ satisfies $(F, X, S)$ if $F \subseteq G \subseteq F \cup X$ 
and $G \setminus F$ is non-decreasing for every~$g \in S \cup \{f\}$.

An \emph{$h_i$-index} of~$(F, X, S)$ is a code~$\tuple{i, F, e}$
such that~$\Phi^{h_i}_e = (X \oplus S \oplus C)'$.
Given an $h_i$-condition~$c = (F, X, S)$, we let $\#(c)$ be the number of functions~$g \in S$
such that~$g$ is not constant over~$X$. Note that an $h_{i+1}$-index of~$c$ can be $P$-computed
from an $h_i$-index of~$c$, and that $\#(c)$ can be $h_i$-computed from
an $h_i$-index of~$c$.

\begin{lemma}\label{lem:first-jump-cns-ext}
For every~$n \in \N$ and every~$h_n$-condition $c = (F, X, S)$, there is 
an $h_{n+1}$-extension $d = (E, Y, S)$ of~$c$
such that either~$\#(d) < \#(c)$, or $|E| > |F|$.
Furthermore, an $h_{n+1}$-index of~$d$ may be $P$-computed
from an $h_n$-index of~$c$.
\end{lemma}
\begin{proof}
Pick any~$x \in X$.
Since~$h_{n+1} \gg (X \oplus S \oplus C)'$, one can $h_{n+1}$-computably
decide if there is some~$g \in S$ and some~$u < g(x)$ such that the set~$Y_0 = \{ y \in X : g(y) = u \}$
is infinite, or whether the set~$Y_1 = \{ y \in X : (\forall g \in S)g(y) \geq g(x) \}$ is infinite.
In the first case, the $h_{n+1}$-condition~$d = (F, Y_0, S)$ is an extension of~$c$
such that~$\#(d) < \#(c)$.
In the second case, let $A = \{ y \in Y_1 : f(y) \geq f(x) \}$.
By Theorem~\ref{thm:first-jump-d22}, either we obtain a $h_{n+1}$-jump index of~$Z_0 \oplus X \oplus S \oplus C$,
where~$Z_0$ is an infinite subset of~$A$, or an $h_{n+1}$-jump index of~$Z_1 \oplus X \oplus S \oplus C$,
where~$Z_1$ is an infinite subset of~$Y_1 \setminus A$.
The second case cannot happen since otherwise, $f$ would be of bounded range over~$Z_1$,
and by further applications of Theorem~\ref{thm:first-jump-d22},
one would obtain an infinite set~$H$ over which~$f$ is constant and such that~$(H \oplus C)' \leq_T P$,
contradicting our initial assumption.
We therefore obtain an infinite set~$Z_0 \subseteq A$
such that $h_{n+1} \geq_T (Z_0 \oplus S \oplus C)'$.
The $h_{n+1}$-condition $d = (F \cup \{x\}, Z_0, S)$ is an extension of~$c$ satisfied
the desired property.
\end{proof}

An $h_i$-condition~$c = (F, X)$ \emph{decides} $(G \oplus C)'(e)$ if
either~$\Phi_e^{F \oplus C}(e) \downarrow$, or $\Phi_{e}^{(F \cup H) \oplus C}(e) \uparrow$
for every set~$H \subseteq X$.

\begin{lemma}\label{lem:first-jump-cns-decide}
For every~$n \in \N$, every~$h_n$-condition $c = (F, X, S)$ and every~$e \in \N$, there is 
an $h_{n+1}$-extension $d$ of~$c$
such that either~$\#(d) < \#(c)$, or $d$ decides $(G \oplus C)'(e)$.
Furthermore, an $h_{n+1}$-index of~$d$ may be $P$-computed
from an $h_n$-index of~$c$, and one can $P$-computably decide which case applies.
\end{lemma}
\begin{proof}
Let~$\Ccal$ be the $\Pi^{0,X \oplus S \oplus C}_1$ class of all functions~$p : \N \to \N$
bounded by~$b$, such that
$p(x) \leq p(y)$ for every~$x \in F$ and~$y \in X$, and for every finite set~$E \subseteq X$
non-decreasing for every~$g \in S \cup \{p\}$ simultaneously, $\Phi^{(F \cup E) \oplus C}(e) \uparrow$.
Since~$h_n \geq_T (X \oplus S \oplus C)'$, one can~$h_n$-decide whether~$\Ccal$ is empty or not.

If~$\Ccal$ is empty, then in particular~$f \not \in \Ccal$.
Unfolding the definition, there is a finite set~$E \subseteq X$
non-decreasing for every~$g \in S \cup \{f\}$, such that~$\Phi^{(F \cup E) \oplus C)}(e) \downarrow$.
As in Lemma~\ref{lem:first-jump-cns-ext}, one can $h_{n+1}$-computably decide
whether there is some~$g \in S$ and some~$u < \max \{g(x) : x \in E\}$ such that the set~$Y_0 = \{ y \in X : g(y) = u \}$
is infinite, or whether the set~$Y_1 = \{ y \in X : (\forall g \in S)(\forall x \in E)g(y) \geq g(x) \}$ is infinite.
In the first case, the $h_{n+1}$-condition~$d = (F, Y_0, S)$ is an extension of~$c$
such that~$\#(d) < \#(c)$.
In the second case, let $A = \{ y \in Y_1 : (\forall x \in E) f(y) \geq f(x) \}$.
Still by the same argument as in Lemma~\ref{lem:first-jump-cns-ext}, 
one can $P$-computably find an infinite set~$Z_0 \subseteq A$
such that~$h_{n+1} \geq_T (Z_0 \oplus S \oplus C)'$. The $h_{n+1}$-condition~$(F \cup E, Z_0, S)$
is an extension of~$c$ forcing~$(G \oplus C)'(e) = 1$.

If~$\Ccal \neq \emptyset$, then by the relativized low basis theorem~\cite{Jockusch197201},
one can $h_n$-computably pick some~$g \in \Ccal$ such that~$h_n \geq_T (g \oplus X \oplus S \oplus C)'$.
The $h_{n+1}$-condition~$(F, X, S \cup \{g\})$ is an extension of~$c$ forcing~$(G \oplus C)'(e) = 0$.
\end{proof}

Using Lemma~\ref{lem:first-jump-cns-ext} and Lemma~\ref{lem:first-jump-cns-decide}, build an infinite
$P$-computable decreasing sequence of tuples~$(\emptyset, \omega, \emptyset) = (F_0, X_0, S_0) \geq (F_1, X_1, S_1) \geq \dots$
such that for every~$s \in \N$,
\begin{itemize}
	\item[(i)] $(F_s, X_s, S_s)$ is an $h_n$-condition for some~$n \in \N$
	\item[(ii)] $|F_s| \geq s$
	\item[(iii)] $(F_s, X_s, S_s)$ decides $(G \oplus C)'(e)$
\end{itemize}
The set~$G = \bigcup_s F_s$ is an infinite non-decreasing subsequence for~$f$
such that~$(G \oplus C)' \leq_T P$. This completes the proof of Theorem~\ref{thm:first-jump-cns}.
\end{proof}

\begin{corollary}
Every computable instance of~$\mathsf{CNS}$ admits a low${}_2$ solution.
\end{corollary}
\begin{proof}
Let~$f$ be a computable instance of~$\mathsf{CNS}$.
By the relativized low basis theorem~\cite{Jockusch197201}, 
there is a set~$P \gg \emptyset'$ such that~$P' \leq_T \emptyset''$.
By Theorem~\ref{thm:first-jump-cns}, there is an infinite non-decreasing subsequence~$G$
for~$f$, such that~$G' \leq_T P$. In particular, $G'' \leq_T P' \leq_T \emptyset''$,
so~$G$ is low${}_2$.
\end{proof}

\section{Summary and open questions}

In this section, we summarize the known relations between $\mathsf{CNS}$, $\mathsf{LNS}$,
and existing principles in reverse mathematics. We also state some remaining open questions.
In Figure~\ref{fig:summary}, and plain arrow from~$\Psf$ to~$\Qsf$
means that~$\Psf$ implies~$\Qsf$ over~$\rca$, while a dotted arrow stands
for an open implication.

\begin{center}
\begin{figure}[htbp]\label{fig:summary}
\begin{tikzpicture}[x=2cm, y=1.5cm, 
		node/.style={minimum size=2em, inner sep=3pt},
		arrow/.style={very thick,->}
]

	\node[draw, minimum width = 7cm, minimum height = 6.5cm] at (3, 3.2) {};
	\node[node] (aca) at (3, 5) {$\aca$};
	\node[node] (wkl) at (4, 4) {$\wkl$};
	\node[node] (wwkl) at (4, 3) {$\wwkl$};
	\node[node] (dnr) at (3.5, 1.5) {$\dnr$};
	\node[node] (rt22) at (2, 4) {$\rt^2_2$};
	\node[node] (srt22) at (2, 3) {$\srt^2_2$};
	\node[node] (cns) at (3, 4) {$\mathsf{CNS}$};
	\node[node] (lns) at (3, 3) {$\mathsf{LNS}$};
	\node[node] (opt) at (2.5, 1.5) {$\opt$};

	\draw[arrow] (aca) -- (wkl);
	\draw[arrow] (aca) -- (cns);
	\draw[arrow] (aca) -- (rt22);
	\draw[arrow] (rt22) -- (srt22);
	\draw[arrow] (cns) -- (lns);
	\draw[arrow] (wkl) -- (wwkl);
	\draw[arrow] (srt22) -- (opt);
	\draw[arrow] (srt22) -- (dnr);
	\draw[arrow] (lns) -- (opt);
	\draw[arrow] (lns) -- (dnr);
	\draw[arrow] (wwkl) -- (dnr);
	\draw[arrow] (cns) -- (srt22);

	\draw[arrow, dotted] (rt22) -- (lns);
	\draw[arrow, dotted] (srt22) -- (lns);
	\draw[arrow, dotted] (cns) -- (rt22);

\end{tikzpicture}
\caption{Non-decreasing subsequences in reverse mathematics}
\end{figure}
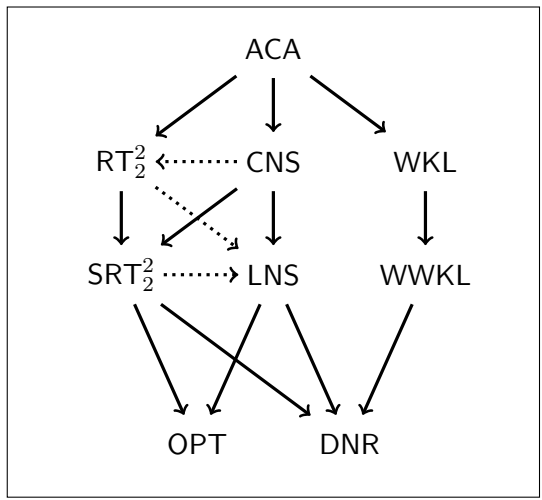
\end{center}

We wonder about the two remaining open implications
between Ramsey's theorem for pairs and the non-decreasing sequence statements.

\begin{question}
Does~$\mathsf{CNS}$ imply $\rt^2_2$ over~$\rca$?
\end{question}

\begin{question}
Does~$\rt^2_2$ imply~$\mathsf{LNS}$ over~$\rca$?
\end{question}

The same questions hold for $\omega$-models and over computable reducibility.

\vspace{0.5cm}

\noindent \textbf{Acknowledgements}.
The author is funded by the John Templeton Foundation (`Structure and Randomness in the Theory of Computation' project). The opinions expressed in this publication are those of the author(s) and do not necessarily reflect the views of the John Templeton Foundation.

\vspace{0.5cm}

\bibliographystyle{plain}

\end{document}